\documentclass{amsart}
\usepackage[all]{xypic}
\usepackage{color}

\theoremstyle{plain}
\newtheorem{theorem}{Theorem}
\newtheorem{lemma}[theorem]{Lemma}

\newtheorem{proposition}[theorem]{Proposition}

\theoremstyle{conjecture}

\theoremstyle{question}

\theoremstyle{definition}
\newtheorem{definition}{Definition}

\newtheorem{examples}{Examples}
\newtheorem{example}{Example}
\theoremstyle{remark}
\newtheorem{remark}{Remark}[section]
\theoremstyle{remarks}
\newtheorem{remarks}{Remarks}[section]

\begin{document}
\title{Complex Groups and Root Subgroup Factorization}

\author{Doug Pickrell}
\email{pickrell@math.arizona.edu}

\maketitle

\begin{abstract} Root subgroup factorization is a refinement of triangular (or LDU) factorization. For a complex reductive Lie group, and a choice of reduced factorization of the longest Weyl group element, the forward map from root subgroup coordinates to triangular coordinates is polynomial. We show that the inverse is rational. There is an algorithm for the inverse (involving LDU factorization), and a related explicit formula for Haar measure in root subgroup coordinates. In classical cases there are preferred reduced factorizations of the longest Weyl group elements, and conjecturally in these cases there are closed form expressions for root subgroup coordinates.

\end{abstract}

\{2000 Mathematics Subject Classifications:  22E67\}

\setcounter{section}{-1}

\section{Introduction}

Let $G$ denote a connected complex Lie group with reductive Lie algebra $\mathfrak g$, e.g. $G=GL(n,\mathbb C)$.
Fix a linear triangular decomposition ${\mathfrak g}={\mathfrak n}^-+{\mathfrak h}+{\mathfrak n}^+$. If $H=\exp({\mathfrak h})$, $N^\pm=\exp({\mathfrak n}^\pm)$, and $W:=N_G(H)/H$ (the Weyl group), then there is a disjoint (Birkhoff) decomposition of $G$
\begin{equation}\label{birkhoff}
G=\bigsqcup_{w\in W}\Sigma_w^G \  \quad \mathrm{    where   }\ \quad  \Sigma_w^G=N^-wHN^+
\end{equation} We are primarily interested in the component corresponding to $w=1$. The component $\Sigma_1^G$ is Zariski open in $G$ and consists of group elements which have a triangular (or LDU)
factorization, $g=ldu$, where $l\in N^-$, $d\in H$, and $u\in N^+$. The LDU factorization is unique, and the factors are rational functions of $g$ (see Subsection \ref{concludingcomment}).

Root subgroup coordinates for a Zariski open subset of $\Sigma^G_1$ (and hence of $G$) involves an additional choice of a reduced factorization for the longest Weyl group element,
$$w_0=r_{\mathbf n}...r_1$$
in terms of reflections corresponding to simple positive roots. Given this factorization, there is a forward map (which depends on some minor additional
choices that we will temporarily suppress)
\begin{equation}\label{forwardmap}  F:(\mathbb C^{2})^{\mathbf n} \to N^-N^+:(\zeta_1,...,\zeta_{\mathbf n})\to lu\end{equation}
where
$$lu=i_{\tau_{\mathbf n}}(g(\zeta_{\mathbf n}))...i_{\tau_1}(g(\zeta_1)), \qquad \qquad
g(\zeta)=\left(\begin{matrix}1&\zeta^+\\\zeta^-&1+\zeta^-\zeta^+\end{matrix}\right)$$
and the $i_{\tau}:SL(2,\mathbb C)\to G$ are root homomorphisms. In addition to showing this forward map $F$ is well-defined and polynomial,
we will show that the inverse, properly understood, is rational.

\begin{example}\label{A2} For $G=GL(3,\mathbb C)$ and the factorization of $w_0=s_1s_2s_1$, where $s_i$ transposes $i$ and $i+1$, the forward
map $F$ sends $(\zeta_1,\zeta_2,\zeta_3)\in (\mathbb C^2)^3$ to
\[
\left(\begin{matrix}1 &0&0 \\0 &1 &\zeta_3^+ \\0 &\zeta_3^- & 1+\zeta_3^-\zeta_3^+ \end{matrix}\right)
\left(\begin{matrix}1 &0&\zeta_2^+ \\0&1&0\\\zeta_2^- &0&1+\zeta_2^-\zeta_2^+  \end{matrix}\right)
\left(\begin{matrix}1 &\zeta_1^+&0 \\\zeta_1^- &1+\zeta_1^-\zeta_1^+ &0 \\0 &0 & 1 \end{matrix}\right)\]
$$=\left(\begin{matrix}1 &0&0 \\\zeta_1^-+\zeta_2^-\zeta_3^+ &1 &0\\\zeta_2^-+\zeta_1^-\zeta_3^-+\zeta_2^-\zeta_3^-\zeta_3^+ &\zeta_3^- & 1 \end{matrix}\right)
\left(\begin{matrix}1 &\zeta_1^+&\zeta_2^+ \\0 &1 &-\zeta_1^-\zeta_2^+ +\zeta_3^+ \\0 &0 & 1 \end{matrix}\right)$$
The inverse map is given by
$$\zeta^-_{3} = l_{3, 2},\quad \zeta^-_{2} =l_{3, 1} -l_{2, 1}l_{3, 2},\quad \zeta_1^-= \frac{l_{2, 1}-l_{3, 1}u_{2, 3}+l_{2,1}l_{3, 2}u_{2, 3}}{1+l_{3, 1}u_{1, 3}-l_{2, 1}l_{3, 2}u_{1, 3}}, $$
$$\zeta^+_{3} = \frac{l_{2, 1}u_{1, 3}+u_{2, 3}}{1+l_{3, 1}u_{1, 3}-l_{2, 1}l_{3, 2}u_{1, 3}}, \quad\zeta^+_{2} = u_{1, 3}, \quad \zeta^+_{1} = u_{1, 2}$$
where $l_{ij}$ and $u_{ij}$ refer to standard coordinates for $l$ and $u$ respectively.

It is useful to factor $l$ and $u$ using `ordered exponential coordinates',
$$l=\left(\begin{matrix}1 &0&0 \\0 &1 &0\\0 &\zeta_3^- & 1 \end{matrix}\right)
 \left(\begin{matrix}1 &0&0 \\0 &1 &0\\\zeta_2^-&0 & 1 \end{matrix}\right)
 \left(\begin{matrix}1 &0&0 \\\zeta_2^-\zeta_3^++\zeta_1^- &1 &0\\0 &0 & 1 \end{matrix}\right)$$
and
 $$u=\left(\begin{matrix}1 &0&0 \\0 &1 &-\zeta_1^-\zeta_2^+ +\zeta_3^+ \\0 &0 & 1 \end{matrix}\right)
\left(\begin{matrix}1 &0&\zeta_2^+ \\0 &1 &0 \\0 &0 & 1 \end{matrix}\right)
\left(\begin{matrix}1 &\zeta_1^+&0\\0 &1 &0 \\0 &0 & 1 \end{matrix}\right)$$
In these coordinates the inverse map is given by
$$\zeta^-_{3} = l_{3},\quad \zeta^-_{2} =l_{2},\quad \zeta_1^-= \frac{l_{1}-l_{2}u_{3}}{1+l_{2}u_{2}}, $$
$$\zeta^+_{3} = \frac{u_2l_1+u_3}{1+l_2u_2},\quad \zeta^+_{2} = u_{2}, \quad \zeta^+_{1} =u_1 $$
To be more precise, $F$ induces a bijective correspondence
$$\{\zeta: 1+\zeta^-_2\zeta^+_2\ne 0\} \leftrightarrow \{(l,u)\in N^-\times N^+:1+l_2u_2\ne 0\}$$

There is a second reduced factorization $w_0=s_2s_1s_2$. In standard coordinates the exceptional set (i.e. the pairs $l,u$ not in the image of the forward map) is given by $1-l_{3,1}u_{1,3}-l_{3,1}u_{1,2}u_{2,3}=0$. The intersection of the exceptional sets corresponding to the two factorizations is a nonvacuous complex $4$ dimensional submanifold. As a consequence, even in this simplest example, root subgroup coordinate charts do not cover $\Sigma_1^G$.
\end{example}

\begin{example}\label{A3} For $G=GL(4,\mathbb C)$, a preferred ordering of the positive roots (a rationale for this preference is explained in Section \ref{canonicalorderings}) is indicated as follows:
$$\left(\begin{matrix} 0& \tau_1&\tau_2&\tau_4\\ & 0 &\tau_3&\tau_5\\ & & 0 & \tau_6\\ &  & &0\end{matrix}\right) $$
In ordered exponential coordinates for $l$ and $u$, there are six trivial equations,
$$u_1=\zeta_1^+,\qquad u_2=\zeta_2^+,\qquad u_4=\zeta_4^+ $$
$$l_4=\zeta_4^-,\qquad l_5=\zeta_5^-,\qquad l_6=\zeta_6^- $$
and (after making substitutions) six nontrivial equations,
$$u_3=\zeta_3^+-\zeta_1^-u_2,\qquad
u_5=\zeta_5^+-\zeta_1^-u_4-\zeta_2^-\zeta_3^+u_4,\qquad
u_6=\zeta_6^+-\zeta_2^-u_4-\zeta_3^-\zeta_5^+ $$
$$l_1=\zeta_1^-+\zeta_2^-\zeta_3^++l_4\zeta_5^+,\qquad l_2=\zeta_2^-+l_4\zeta_6^+-\zeta_3^-l_4\zeta_5^+,\qquad l_3=\zeta_3^-+l_5\zeta_6^+
$$
The fact that the equations are multilinear reflects the fact that $G$ is simply laced.
In this case $F$ induces a bijective correspondence between $\{\zeta:1+\zeta^-_k\zeta^+_k\ne 0, k=2,4,5\}$
and the complement in $N^-\times N^+$ of the vanishing set of the three denominators which appear in the formula for the rational inverse.
\end{example}

In general (see Theorem \ref{detjacobian})
\begin{equation}\label{detformula}det(\partial F)=\prod_k (1+\zeta^-_k\zeta^+_k)^{\delta(h_{\tau_k})-1}\end{equation}
where $\delta$ is half the sum of the positive roots and $h_{\tau}$ is the coroot corresponding to a root $\tau$. As the examples above illustrate, we will show that $F$ induces a bijective correspondence between $\{\zeta:det(\partial F)\ne 0\}$, i.e. the set of points where $F$ is locally bijective, and the complement in $N^-\times N^+$ of the vanishing set of the denominators that appear for the rational inverse of $F$. It is possible that this follows from general principles. In connection with this, it is interesting to recall (as Jack Hall pointed out to me) the Ax-Grothendieck theorem (and the related Jacobian conjecture), with a refinement and from a point of view due to Rudin, which asserts that an injective polynomial map $\mathbb C^n \to\mathbb C^n$ is bijective and has a polynomial inverse (see \cite{Ax},\cite{Groth},\cite{Rudin}). In our case, by examining an algorithm for solving for the $\zeta$ variables (which involves an LDU factorization for $g^{-t}$), we will simultaneously show that $F$ is injective on $\{\zeta:det(\partial F)\ne 0\}$ and that $F^{-1}$ is rational.

Given any ordering of the positive roots, $\tau_{1}$,...,$\tau_{\mathbf n}$, and choices of corresponding
root homomorphisms, one can consider a forward map
$$ (\mathbb C^{2})^{\mathbf n} \to G:(\zeta_1,...,\zeta_{\mathbf n})\to g$$
where
$$g=i_{\tau_{\mathbf n}}(g(\zeta_{\mathbf n}))...i_{\tau_1}(g(\zeta_1))$$
In general the image is not contained in $N^-HN^+$, i.e. the LDU coordinates are rational functions of the $\zeta$ variables.
It seems very plausible that this forward map is polynomial if and only if the ordering of the roots comes from a reduced factorization of $w_0$. Similarly, for a general ordering of the roots, the inverse is algebraic (but in general multivalued), and it seems
very plausible that the inverse is rational if and only if the ordering comes from a reduced factorization of $w_0$. For two characterizations of orderings which correspond to reduced factorizations, see \cite{Kras2}.

How special are the orderings that arise from factorizations of $w_0$? In the case of $GL(n,\mathbb C)$, there are $\left(\begin{matrix}n\\2\end{matrix}\right)$ positive roots, hence $\left(\begin{matrix}n\\2\end{matrix}\right)!$ orderings of the positive roots; Stanley \cite{stanley} showed that there are
$$\frac{\left(\begin{matrix}n\\2\end{matrix}\right)!}{1^{n-1}3^{n-2}..(2n-3)^1}$$
reduced factorizations of $w_0$. Kraskiewicz \cite{Kras1} showed that this formula has a representation theoretic interpretation. In the case of $B_r$ and $C_r$, there are $r^2$ positive roots, and Stanley conjectured, and Kraskiewicz \cite{Kras1} proved, that there are
$$\frac{r^2!}{\prod_{i=1}^n(2i-1)^{n-i}\prod_{j=0}^{n-3}\left(\prod_{k=1}^{n-j-2}(2(j+2k))\right)}$$ reduced factorizations of $w_0$.

\subsection{Notes}

There is a substantial literature on (a number of slightly different versions of) root subgroup factorization for a unitary form $U$ of $G$. Some of the main references are Kac-Peterson \cite{KP} (following seminal ideas of Steinberg on the structure of finite groups of Lie type), Bott-Samelson \cite{BottSamelson} (on the desingularization of Schubert varieties), Soibelman \cite{Soibelman} and Lu \cite{Lu} (Poisson geometry). For real noncompact groups, see \cite{CP}. One motivation for returning to this topic in a complex setting is the existence of a generalization to loop groups which has nontrivial consequences for the calculation of Toeplitz determinants, see \cite{BEP}. This connection with loop groups explains our motivation for referring to the decomposition (\ref{birkhoff}) as a Birkhoff decomposition, to distinguish it from the more algebraically natural Bruhat decomposition,
$$G=\bigsqcup_{W} B^+wB^+$$ We are trying to gain insight from this finite dimensional setting that will help us in more analytically challenging settings.

\subsection{Plan of the Paper}

In Section \ref{background} we recall some basic facts about triangular structures for Lie algebras and groups, and factorization for Weyl group elements. In Section \ref{polynomial} we show that the forward map is polynomial. In Section \ref{rational} we show that the inverse is rational. We will actually consider two different algorithms for solving for the inverse, one of which we can justify. In Section \ref{haar} we discuss a formula for Haar measure in root subgroup coordinates. Finally in Section \ref{canonicalorderings} we discuss canonical orderings in classical cases. For these canonical orderings, it appears that there are reasonable closed form expressions for root subgroup coordinates, but the nature of these formulas is unclear (as Hermann Flaschka suggested to me, it is possible these are related to a cluster algebra).

\subsection{Acknowledgement} I thank Jeremy Roberts, who assisted me with a number of experiments related to Section \ref{canonicalorderings}, and I thank Jack Hall for pointing me to the Ax-Grothendieck theorem and related discussions.

\section{Notation and Background\label{background}}

Let $\mathfrak a:=\mathfrak h_{\mathbb R}$, $\mathfrak t=i\mathfrak h_{\mathbb R}$, and write $h\in H$ as $h=ma$ relative to the global decomposition $H=TA$ ($T$ is the maximal torus of a unitary form $K$ for $G$, but $K$ will not play a significant role in this paper).

For each positive root $\alpha$ let $h_{\alpha}\in{\mathfrak a}$ denote the associated coroot (satisfying $\alpha(h_{\alpha})=2$).
If $\gamma$ is a simple positive root, fix a root homomorphism $\iota_{\gamma}\colon \mathfrak{sl}(2,{\mathbb C})\to{\mathfrak g}_{-\gamma}\oplus\mathbb C h_{\gamma}\oplus{\mathfrak g}_{\gamma}$ which maps the standard basis for $sl(2,\mathbb C)$
to $f_{\gamma}\in {\mathfrak g}_{-\gamma}$, $h_{\gamma}$, and $e_{\gamma}\in {\mathfrak g}_{\gamma}$  with $\tau(e_{\gamma})=-f_{\gamma}$. We denote the corresponding group homomorphism by the same symbol.
For each simple positive root $\gamma$, we use the group homomorphism to set
\begin{equation}\label{defn_of_r_gamma}
\mathbf r_{\gamma}=\iota_{\gamma}\begin{pmatrix} 0 & i \\ i & 0\end{pmatrix}\in N_{G}(H)
\end{equation}
and obtain a specific representative for the associated simple reflection $r_{\gamma}$ in the Weyl group $W:=N_{G}(H)/H$ (We will adhere to the convention of using boldface letters to denote representatives of Weyl group elements).

\begin{remark}\label{roothomomorphisms}
Throughout this paper we regard $\iota_{\gamma}$, $f_{\gamma}$ and $e_{\gamma}$ corresponding to a simple positive root $\gamma$ as fixed.
If $\eta$ is another positive root, then there is a Weyl group element $w$ such that $\eta=w\cdot \gamma$; by choosing a representative
$\mathbf w \in N_{G}(H)$ for $w$, we obtain a homomorphism $\iota_{\eta}(\cdot)=\mathbf w\iota_{\gamma}(\cdot)\mathbf w^{-1}$, and we denote the image of the standard basis by $f_{\eta}$, $h_{\eta}$ (which is consistent with the previous definition of the coroot corresponding to $\eta$), and $e_{\eta}$. These objects depend on the choice of $w$ and its representative $\mathbf w$.\end{remark}

By definition, the Birkhoff decomposition of $G$ relative to the triangular decomposition ${\mathfrak g}={\mathfrak n}^-+{\mathfrak h}+{\mathfrak n}^+$ is
\begin{equation}
G=\bigsqcup_{W} \Sigma^{G}_{
w} \text{ where }\Sigma^{G}_{ w}=N^- w B^+.
\end{equation}
If we fix a representative $\mathbf w\in N_{G}(H)$ for $w\in W$,
then each $g\in\Sigma^{G}_{w}$ can be factored uniquely as
\begin{equation}\label{birkhoff2}
g=l\mathbf {\mathbf w} ma u,\text{ with } l\in N^-\cap w N^-w^{-1},\ ma\in
TA, \text{ and }u\in N^+.
\end{equation}
This defines functions $l\colon \Sigma_{w}^{G}\to N^-\cap wN^-w^{-1}$, $m\colon \Sigma_{w}^{G}\to T$, $a\colon \Sigma_{w}^{G}\to A$, and $u\colon \Sigma_{w}^{G}\to N^+$. For fixed $m_0\in T$, the subset $\{g\in \Sigma^{G}_{ w}: m(g)=m_0\}$ is a stratum (topologically an affine space). It is therefore sensible and appropriate to refer to $\Sigma^{G}_{w}$ as the ``isotypic component of the Birkhoff decomposition of $G$ corresponding to $w\in W$."
However we may occasionally lapse into referring to  $\Sigma^{G}_{
w}$ as the ``Birkhoff stratum corresponding to $w$."

\subsection{Weyl Group and Orderings of Roots}

We recall a number of standard facts about the Weyl group $W$, see e.g. the Appendix of \cite{varadarajan}.
The Weyl group $W$ is generated by the simple reflections. For $w\in W$, $l(w)$ denotes the length of a minimal factorization
for $w$ in terms of simple reflections; in reference to the natural action of $W$ on roots, $l(w)$ is the number of positive roots which are flipped to negative roots by $w$. There is a unique element $w_0\in W$ of longest length; it maps the set of all positive roots (positive simple roots) to the set of negative roots (negative simple roots, respectively), hence its length is the number of positive roots (or the complex dimension of $G/B^+$).

\begin{lemma}\label{weylfactorization} Fix $w\in W$.  Choose a sequence of simple positive roots $\gamma_j$ and associated simple reflections $r_j$ in the following way: (1) choose $\gamma_1$ such that $w\cdot
\gamma_1>0$; (2) choose $\gamma_2$ such that $wr_1\cdot
\gamma_2>0$; (3) choose $\gamma_3$ such that $w r_1r_2\cdot
\gamma_3>0$, and so on, where $r_j$ is the simple reflection
corresponding to $\gamma_j$. This procedure terminates after $\mathbf n=l(w_0)-l(w)$ steps.  Let $\tau_j=r_1...r_{j-1}\cdot\gamma_j$. Then the $\tau_j$ are the
positive roots which are mapped to positive roots by
$w$ and $r_{\mathbf n}\dots r_1$ is a reduced decomposition of $w'=w_0w$.
\end{lemma}

\begin{proof} This is more commonly expressed in the following way: Set $w'=w_0w$. Choose a
reduced decomposition $w'=r_{\mathbf n}...r_1$ where ${\mathbf n}=l(w_0)-l(w)=l(w')$.
Then for $n\le \mathbf n$, the positive roots $\tau_j$ which are mapped to negative roots by $r_n..r_1$ are of the form $\tau_j=r_1...r_{j-1}\cdot\gamma_j$, $j=1,..,n$; see Theorem 4.15.10 in the Appendix of \cite{varadarajan}.
\end{proof}

We are primarily interested in the case $w=1$. Given a reduced factorization $w_0=r_{\mathbf n}...r_1$ in terms of simple reflections, we obtain an ordering of the positive roots $\tau_1,...,\tau_{\mathbf n}$. Note that $\tau_1$ and $\tau_{\mathbf n}=w_0\cdot (-\gamma_{\mathbf n})$ are simple. For a general ordering of the positive roots $\tau_1,...,\tau_{\mathbf n}$, the Lemma gives an algorithm for determining whether this ordering comes from a factorization of $w_0$: (1) $\tau_1=\gamma_1$ must be simple, (2) there exists simple positive $\gamma_2$ such that $r_{1}\cdot
\gamma_2=\tau_2$; (3) there exists simple positive $\gamma_3$ such that $r_{1}r_{2}\cdot
\gamma_3=\tau_3$, and so on, where $r_j$ is the simple reflection corresponding to $\gamma_j$.

\begin{examples} \label{orderingexamples}Consider a reduced factorization $w_0=r_{\mathbf n}...r_1$ and the associated ordering of the positive roots $\tau_1,...,\tau_{\mathbf n}$.

(1) The ordering associated to the (reversed) reduced factorization $w_0=r_1...r_{\mathbf n}$ is $\gamma_{\mathbf n},r_{\mathbf n}\cdot \gamma_{\mathbf n-1},...  $

(2)  The ordering associated to the (conjugated) reduced factorization $w_0=r_{\mathbf n}^{w_0}...r_1^{w_0}$ is $w_0\cdot(-\gamma_{1}), w_0\cdot(-\tau_2),...w_0\cdot(-\tau_{\mathbf n})$

(3) The associated ordering to the (conjugated and reversed) reduced factorization $w_0=r_{1}^{w_0}...r_{\mathbf n}^{w_0}$ is the reversed ordering $\tau_{\mathbf n},...,\tau_1$.

\end{examples}

For examples of canonical factorizations in classical cases, see Section \ref{canonicalorderings}.
For characterizations of orderings of positive roots which come from factorizations of $w_0$, see \cite{Kras2}.

\section{The Forward Map is Polynomial}\label{polynomial}

Given a pair $\zeta=(\zeta^-,\zeta^+)\in \mathbb C^2$, define
\begin{equation}\label{factor_plus}
g(\zeta):=\begin{pmatrix}1&\zeta^+\\\zeta^-&1+\zeta^-\zeta^+
\end{pmatrix}=\begin{pmatrix}1&0\\\zeta^-&1
\end{pmatrix}\begin{pmatrix}1&\zeta^+\\0&1
\end{pmatrix} \in {\mathrm SL}(2,\mathbb C)
\end{equation}

\begin{proposition}\label{forwardmap} Fix $w\in W$ and a representative $\mathbf w\in N_{G}(H)$ for $w$, then determine positive simple roots $\gamma_1,\dots,\gamma_{\mathbf n}$ with associated simple reflections $r_1,\dots,r_{\mathbf n}$, and positive roots $\tau_1,\dots,\tau_{\mathbf n}$ as in Lemma \ref{weylfactorization}.  Set ${\mathbf w}_j'=\mathbf r_j...\mathbf r_1$ and $\iota_{\tau_j}(g)=({\mathbf w}_{j-1}')^{-1}\iota_{\gamma_j}(g){\mathbf w}_{j-1}'$ for each $g\in \mathrm{SL}(2,{\mathbb C})$.

(a) If $h\in H$, $(\zeta_1,\dots,\zeta_{\mathbf n})\in{\mathbb C}^{2\mathbf n}$, and
\[
g:=\mathbf w \,\iota_{\tau_{\mathbf n}}(g(\zeta_{\mathbf n}))..\iota_{\tau_1}(g(\zeta_1)) h
\]
then
 \[g\in \Sigma_w \]

(b) In reference to (\ref{birkhoff2}), $ma(g)=h$, and in exponential coordinates for $N^{\pm}$, $l$ and $u$
are polynomial functions of the $\zeta$ variables.

\end{proposition}

\begin{remark} There is a notational subtlety in the statement of the lemma.  It can happen that for some $j$, $\tau_j=\gamma$ for some simple positive root $\gamma$.  Thus, the root homomorphism $\iota_{\tau_j}$ may depend on $j$, i.e., on the ordering of the roots (we commented on this earlier in Remark \ref{roothomomorphisms}).

\end{remark}

\begin{proof} We can suppose that $h=1$.

We first calculate the triangular decomposition for the partial product
\[
g^{(n)}:=\iota_{\tau_n}(g(\zeta_n))...\iota_{\tau_1}(g(\zeta_1))
\]
by induction on $n\le \mathbf n$. In the calculation, we will use the notation introduced in Remark \ref{roothomomorphisms}.
Note that since $\tau_j=(w_{j-1}')^{-1}\cdot \gamma_j>0$ and $\iota_{\tau_j}$ preserves triangular factorizations,
\begin{eqnarray*}
\iota_{\tau_j}(g(\zeta_j)) & = & \iota_{\tau_j}(
\begin{pmatrix}1&0\\\zeta_j^-&1
\end{pmatrix})\iota_{\tau_j}(
\begin{pmatrix}1&\zeta_j^+\\0&1
\end{pmatrix}) \\
& = & \exp(\zeta_j^-
f_{\tau_j})(\mathbf
w_{j-1}')^{-1}\exp(\zeta_j^+ e_{\gamma_j})\mathbf w_{j-1}'
\end{eqnarray*}
is a triangular factorization.

Suppose that $n=2$ (this case illustrates the core issues). Then
\begin{equation}\label{n=2case1}
g^{(2)}=\exp(\zeta_2^- f_{\tau_2})\mathbf
r_1^{-1}\exp(\zeta_2^+ e_{\gamma_2})\mathbf r_1 \exp(\zeta_1^-
f_{\gamma_1})\exp(\zeta_1^+
e_{\gamma_1})
\end{equation}
To obtain a triangular factorization, we must reorder the middle four terms in this product. First write
\begin{equation}\label{n=2_2}
 \mathbf r_1^{-1}\exp(\zeta_2^+ e_{\gamma_2})\mathbf r_1 \exp(\zeta_1^- f_{\gamma_1})=\mathbf r_1^{-1} \exp(\zeta_2^+ e_{\gamma_2})
\exp(\zeta_1^- e_{\gamma_1})\mathbf r_1 \end{equation}
Now $\gamma_1$ and $\gamma_2$ are positive roots, and $\gamma_1$ is the unique positive root mapped to a negative root
by $r_1$. Thus (\ref{n=2_2}) equals
$$ \mathbf r_1^{-1} \exp(\zeta_1^-
e_{\gamma_1})\widetilde u \mathbf r_1, \text{ where }\widetilde u=\exp(\zeta_2^+e_{\gamma_2}+\sum_{j\ge 1}\frac{1}{j!}\zeta_2^+(-\zeta_1^-ad(e_{\gamma_1}))^j(e_{\gamma_2}))$$
and $\tilde u\in N^+\cap r_1^{-1} N^+ r_1$ because the root vectors $(ad(e_{\gamma_1}))^j(e_{\gamma_2})$, $j\ge 0$, (which are not zero) correspond to positive roots which are mapped to positive roots by $r_1$: these roots have increasing height (as $j$ increases), $\gamma_2$ is the only one of height one,
and $\gamma_2\ne \gamma_1$. Thus (\ref{n=2_2}) equals
\begin{equation}\label{n=2_3}\exp(\zeta_1^-
f_{\gamma_1})\mathbf u, \quad (\text{for some }  \mathbf u\in
N^+), \end{equation}
and
\begin{equation}\label{n=2_4}\log(\mathbf u)=\sum_{j\ge 0}\frac{1}{j!}\zeta_2^+Ad(\mathbf r_1)^{-1}(-\zeta_1^-ad(e_{\gamma_1}))^j(e_{\gamma_2})\end{equation}

Insert this calculation into (\ref{n=2case1}). We then see
that $g^{(2)}$ has a triangular factorization $g^{(2)}=l^{(2)}u^{(2)}$, where
\begin{equation}\label{2ndcase1}
l^{(2)} =\exp(\zeta_2^-f_{\tau_2}) \exp(\zeta_1^- f_{\tau_1})= \exp(\zeta_2^-
f_{\tau_2}+\zeta_1^- f_{\tau_1})\end{equation}
(the second equality follows from the fact that a two dimensional nilpotent algebra is abelian), and
$$u^{(2)}=\mathbf u\exp(\zeta_1^+e_{\tau_1}) =\exp(\sum_{j\ge 0}\frac{1}{j!}\zeta_2^+Ad(\mathbf r_1)^{-1}(-\zeta_1^-ad(e_{\gamma_1}))^j(e_{\gamma_2})))\exp(\zeta_2^+e_{\tau_2})\exp(\zeta_1^+e_{\tau_1})   $$
For the sake of completeness, note that in simply laced cases, the terms in the sum will vanish for $j>1$, and hence the terms are multilinear, and fractions will not appear.
In general $u^{(2)}$ has the form
\begin{equation}\label{n=2case4}\left(\prod_{k>2}^{\leftarrow} \exp(u^{(2)}_k e_{\tau_k}) \right)\exp(\zeta_2^+e_{\tau_2})\exp(\zeta_1^+e_{\tau_1})\end{equation}
To see this note that $\tau_1$ and $\tau_2$ are the positive roots that are mapped to negative roots by $r_2r_1$,
whereas for $j>0$
$$Ad(\mathbf r_2 \mathbf r_1)Ad(\mathbf r_1)^{-1}(ad(e_{\gamma_1}))^j(e_{\gamma_2}))
=Ad(\mathbf r_2)(ad(e_{\gamma_1}))^j(e_{\gamma_2}))
$$ is a positive root vector, because the sole positive root mapped to a negative root is $\gamma_2$.

We now want to formulate the induction step. We assume that $g^{(n-1)}$ has a triangular
factorization $g^{(n-1)}=l^{(n-1)}u^{(n-1)}$ with
\begin{equation}\label{induction1}
l^{(n-1)}=\exp(\zeta_{n-1}^-f_{\tau_{n-1}})\widetilde l \in N^-\cap
(w_{n-1}')^{-1}N^+w_{n-1}'=\exp(\sum_{j=1}^{n-1}\mathbb C
f_{\tau_j}),
\end{equation}
for some $\widetilde l \in N^-\cap (w_{n-2}')^{-1}N^+w_{n-2}'=\exp(\sum_{j=1}^{n-2}\mathbb C f_{\tau_j})$.
We have established this for $n-1=1,2$. We will use induction to show that $g^{(n)}$
has a triangular factorization of the same form, with $n$ in place of $n-1$.

For $n \ge 3$
\begin{eqnarray*}
 g^{(n)} & = & \exp(\zeta_n^- f_{\tau_n})(\mathbf w_{n-1}')^{-1}\exp(\zeta_n^+ e_{\gamma_n})\mathbf w_{n-1}'\exp(\zeta_{n-1}^-
f_{\tau_{n-1}})\widetilde l u^{(n-1)} \\
 & = & \exp(\zeta_n^- f_{\tau_n})(\mathbf w_{n-1}')^{-1}\exp(\zeta_n^+e_{\gamma_n}) \widetilde u\mathbf w_{n-1}'u^{(n-1)},
\end{eqnarray*}
where $\widetilde u= \mathbf w_{n-1}'\exp(\zeta_{n-1}^- f_{\tau_{n-1}})\widetilde l (\mathbf w_{n-1}')^{-1}\in w_{n-1}' N^{-} (w_{n-1}')^{-1} \cap
N^+$. Now factor $\exp(\zeta_n^+e_{\gamma_n}) \widetilde
u\in N^+$ as $\widetilde u_1\widetilde u_2$, relative to the product decomposition
\[
N^+=\left(N^+\cap w_{n-1}'N^-(w_{n-1}')^{-1}\right)\left(N^+\cap
w_{n-1}'N^+(w_{n-1}')^{-1}\right)
\]
and let
\[
\mathbf l=(\mathbf
w_{n-1}')^{-1}\widetilde u_1 \mathbf w_{n-1}'\in N^- \cap
(w_{n-1}')^{-1} N^+
w_{n-1}'.
\]
Then $g^{(n)}$ has triangular decomposition
\begin{eqnarray*}
g^{(n)} & = & \left(\exp(\zeta_n^- f_{\tau_n})\mathbf
l\right) \left((\mathbf w_{n-1}')^{-1}\widetilde u_2\mathbf w_{n-1}'u^{(n-1)}\right) \\
& = & l^{(n)}u^{(n)}
\end{eqnarray*}

We have now completed the inductive calculation of the triangular decomposition for $g^{(n)}$ for $n\le \mathbf n$. Now suppose that we multiply this triangular decomposition  for $g^{(\mathbf n)}$ on the left by $\mathbf w$ (as in part (a)).
Because the $\tau_j$, $j=1,..,\mathbf n$, are the positive roots which are mapped
to positive roots by $w$, it follows that $l^{(\mathbf n)}$ will be conjugated by $\mathbf w$ into
another element in $N^-$.  It follows that $g$, as defined in part (a), is in $\Sigma^{G}_w$.

For part (b) note that the induction argument above shows that $ma(g)=1$. For $g$ as in the statement of the proposition,
with $h\in H$ appearing on the right, because $H$ normalizes $N^+$, $ma(g)=h$. The groups $N^{\pm}$ are simply connected nilpotent groups, hence the exponential maps $exp:\mathfrak n^{\pm}\to N^{\pm}$ are  polynomial maps with polynomial inverses. This implies the the forward map is polynomial (we will be more precise about the nature of this polynomial map in the next section).

\end{proof}

\section{The Inverse Map is Rational}\label{rational}

From now on we assume that $w=1$.

Fix a reduced factorization of the longest Weyl group element,
$$w_0=r_{\mathbf n}...r_1$$ This determines bases $\{e_{\tau_j}:j=1..\mathbf n\}$ and $\{f_{\tau_j}:j=1..\mathbf n\}$ for $\mathfrak n^{\pm}$, respectively. Because $N^{\pm}$ are simply connected nilpotent Lie groups, the product maps induce isomorphisms of spaces
$$\mathbb C^{\mathbf n} \to N^-: \vec{l} \to l=exp(l_{\mathbf n} f_{\tau_{\mathbf n}})...exp(l_1 f_{\tau_1}) $$
and
$$\mathbb C^{\mathbf n} \to N^+:\vec{u} \to u=exp(u_{\mathbf n} e_{\tau_{\mathbf n}})...exp(u_1 e_{\tau_1}) $$
We will refer to these as ordered exponential coordinates for $l$ and $u$, respectively. For the purposes of this paper, an elementary but crucial fact is that with respect to exponential coordinates for the images, these are polynomial maps with polynomial inverses.

For many purposes it is convenient to think of the index for $l_j$ to be the root $\tau_j$. We will
use the height of the root to associate a weight to the variable.

\begin{definition} The weight of
$\zeta_k^{\pm}$ is $\pm$ the height of the positive root $\tau_k$, and we additively extend this definition to products of such variables. Similarly the weight of $u_j$ and $l_j$ are defined to be $\pm$ height of $\tau_j$, respectively, and we additively extend this definition to products.
\end{definition}

\begin{lemma}\label{workhorse} Consider the map
$$  (\mathbb C^{2})^{\mathbf n} \to N^-N^+:(\zeta_1,...,\zeta_{\mathbf n})\to g=i_{\tau_{\mathbf n}}(g(\zeta_{\mathbf n}))...i_{\tau_1}(g(\zeta_1))$$

(a) The ordered exponential coordinates $u_j$ and $l_j$ can be expressed as (nonhomogeneous) polynomials of the $\zeta$ variables consisting of terms having weight equal to $\pm$ the height of $\tau_j$, respectively.

(b) $l_j-\zeta_j^-$ has no dependence on $\zeta_k^{\pm}$ for $k\le j$; moreover $l_{\mathbf n}=\zeta_{\mathbf n}^-$ and $l_{\mathbf n-1}=\zeta_{\mathbf n-1}^-$.

(c) $u_j-\zeta_j^+$ has no dependence on $\zeta_k^{\pm}$ for $k\ge j$; moreover $u_1=\zeta_1^+$ and $u_2=\zeta_2^+$.

(d) If $\mathfrak g$ is simply laced, then $l_j$ and $u_j$ are multilinear functions of the $\zeta$ variables.
\end{lemma}

\begin{remark}It is essential that we are considering ordered exponential coordinates. Parts (b) and (c) are not true for ordinary
exponential coordinates, or for (in the general linear case) standard coordinates; see the example $G=GL(3,\mathbb C)$. \end{remark}

\begin{proof} We will use the same notation which we developed in the proof of Proposition \ref{forwardmap}, in particular
\[
g^{(n)}:=\iota_{\tau_n}(g(\zeta_n))...\iota_{\tau_1}(g(\zeta_1))=l^{(n)}u^{(n)}
\] for $n\le \mathbf n$. The truncated product $g^{(n)}$ is the result of setting $\zeta_j^{\pm}=0$ for $j>n$. We will prove that the statements in (a), (b) and (d) concerning the $l_j$ are true for $g^{(n)}$ by induction (where in part (b) we keep in mind that $\zeta_j^{\pm}=0$ for $j>n$). After accomplishing this, we use a Cartan involution symmetry (which is broken for $n<\mathbf n$) to show that the statements concerning the $u_j$ also hold.

When $n=2$ (see (\ref{2ndcase1}))
$$l^{(2)} =\exp(\zeta_2^-f_{\tau_2}) \exp(\zeta_1^- f_{\tau_1}) $$
Thus $l^{(2)}_{i}=\zeta^-_i$, $i=1,2$.

Now suppose that the statements in (a), (b) and (d) for $l^{(k)}_j$ hold for $g^{(k)}=l^{(k)}u^{(k)}$, $k<n$, where the second part of (b) is modified to $l^{(k)}_{k}=\zeta_{k}^-$ and
$l^{(k)}_{k-1}=\zeta_{k-1}^-$. The explicit expression above for $l^{(2)}$ establishes this for $k=1,2$.
For $n \ge 3$, $g^{(n)}$ equals

\begin{equation}
\exp(\zeta_n^- f_{\tau_n})(\mathbf w_{n-1}')^{-1}\left(\exp(\zeta_n^+ e_{\gamma_n})\prod_{1\le j<n}^{\leftarrow}\exp(l_j^{(n-1)}e_j')\right)\mathbf w_{n-1}' \prod_{j}^{\leftarrow}\exp(u_j^{(n-1)}e_{\tau_j})\end{equation}
where (to simplify the notation) we have set  $e'_j:=Ad(\mathbf w_{n-1}')(f_{\tau_j})$, which is a root vector corresponding to the positive root $\tau'_j:=-w_{n-1}'\cdot\tau_j$, $j=1,..,n-1$.
Note that there is a restriction on the product expression for $l^{(n-1)}$ (see (\ref{induction1})), but not for $u^{(n-1)}$ (We are not  concerned about the form of $u^{(n)}$).
Following the argument in the proof of Proposition \ref{forwardmap}, we need to factor
\begin{equation}\label{product4}\exp(\zeta_n^+e_{\gamma_n})\prod_{1\le j<n}^{\leftarrow}\exp(l_j^{(n-1)}e_j') \in N^+\end{equation} relative to the decomposition
\begin{equation}\label{groupdecomposition}
N^+=\left(N^+\cap w_{n-1}'N^-(w_{n-1}')^{-1}\right)\left(N^+\cap
w_{n-1}'N^+(w_{n-1}')^{-1}\right)
\end{equation}
i.e. we have to move $\exp(\zeta_n^+e_{\gamma_n})\in N^+\cap
w_{n-1}'N^+(w_{n-1}')^{-1}$ to the right of the product, which is in $N^+\cap w_{n-1}'N^-(w_{n-1}')^{-1}$.
To do this, we need a lemma concerning the commutation relations for these two subalgebras.

There is a filtration of $n^+\cap w_{n-1}'\mathfrak n^-(w_{n-1}')^{-1}$ by subalgebras
\begin{equation}\label{filtration1}\mathbb C e'_1\subset \bigoplus_{1\le j\le 2}\mathbb C e'_j \subset ...\subset \bigoplus_{1\le j< n}\mathbb C e'_j\end{equation} The $j$th subalgebra is spanned by root vectors corresponding to positive roots which are mapped to negative roots
by $(w_j')^{-1}$.

\begin{lemma}\label{techrootlemma} For each $j\le n$, the (direct) sum of $\oplus_{1\le i<j}\mathbb C e'_i$ and $\mathfrak n^+\cap w_{n-1}'\mathfrak n^+(w_{n-1}')^{-1}$ is a subalgebra. Consequently, augmenting (\ref{filtration1}), there is an increasing sequence of subalgebras (each with a direct sum decomposition into two subalgebras)
\begin{equation}\label{filtration2}\mathfrak n_j^{+(n)}:=\left(\bigoplus_{1\le i\le j}\mathbb C e'_i\right)\bigoplus\left(\mathfrak n^+\cap w_{n-1}'\mathfrak n^+(w_{n-1}')^{-1}\right), \qquad j=1,...,n\end{equation} which filters $\mathfrak n^+$.
\end{lemma}

\begin{proof} Consider a root vector $e_{\alpha}\in n^+\cap w_{n-1}'\mathfrak n^+(w_{n-1}')^{-1}$, corresponding to a positive root $\alpha$ such that $(w'_{n-1})^{-1}(\alpha)>0$. We need to show that for $i\le j$, $[e_{\alpha},e_i']\in \mathfrak n_j^{+(n)}$.

We argue by contradiction. Suppose there is $i\le j$ and $k> j$ such that $[e_{\alpha},e'_i]=c\cdot e_k'$, where $c$ is a nonzero constant. In terms of roots this means that $\alpha+\tau_i'=\tau_k'$, or
$$(w_{n-1}')^{-1}(\alpha)-\tau_i=-\tau_k  \text{  or  }   (w_{n-1}')^{-1}(\alpha)+\tau_k=\tau_i$$
Apply $w_{i}'$ to this latter equality of roots. $w_i'(\tau_i)$ is a negative root. $w_i(\tau_k)$ is a positive root because $k>i$. Also $w_{n-1}'(\alpha)$ is a root which is not mapped to a negative root by $w_{n-1}'$, and hence it is not mapped to a negative root by $w_i'$ (the sets of the positive roots mapped to negative roots by the $w_i'$ are nested and increasing). Thus $w_i'$ applied to the RHS is negative and $w_i'$ applied to the LHS is positive, a contradiction. This implies the lemma.

\end{proof}

Modulo $N^+\cap w_{n-1}'N^+(w_{n-1}')^{-1}$ on the right, $(\ref{product4})$ equals
\begin{equation}\label{product5}\prod_{1\le j<n}^{\leftarrow}\exp(\zeta_n^+e_{\gamma_n})\exp(l_j^{(n-1)}e'_j)\exp(-\zeta_n^+e_{\gamma_n}) \end{equation}
\begin{equation}\label{product6}=\prod_{1\le j<n}^{\leftarrow}\exp(l_j^{(n-1)}e'_j+\sum_{k>0}\frac1{k!}l^{(n-1)}_j(\zeta_n^+ ad(e_{\gamma_n})^k(e'_j)) \end{equation}

\begin{remark}\label{simplylacedremark}For later reference, note that if $\mathfrak g$ is simply laced (so the nondiagonal entries of the Cartan matrix are $-1$ or $0$), then the terms in the sum vanish for $k>1$, and with the exception of $G_2$, the terms vanish for $k>2$ (for $G_2$ the terms vanish for $k>3$).\end{remark}

Because $e_{\gamma}\in \mathfrak n^+\cap w_{n-1}'\mathfrak n^+(w_{n-1}')^{-1}$, Lemma \ref{techrootlemma} implies that
$ad(e_{\gamma_n})^k(e'_j)\in \mathfrak n_j^{+(n)}$. Consequently (\ref{product6}) is of the form
\begin{equation}\label{product7}=\prod_{1\le j<n}^{\leftarrow}\exp(l_j^{(n-1)}e'_j+x_j) \end{equation}
where $x_j$ is in the subalgebra $\mathfrak n_j^{+(n)}$ (The point is that $x_j$ is a combination of the root vectors $e_i'$ which are to the right, plus something in $\mathfrak n^+\cap w_{n-1}'\mathfrak n^+(w_{n-1}')^{-1}$, which we want to move to the right). To do the reordering, we proceed term by term, starting from the left. Consider the left most term $\exp(l_{n-1}^{(n-1)}e'_{n-1}+x_{n-1})$.
By induction $l_{n-1}^{(n-1)}=\zeta^-_{n-1}$. Relative to the decomposition (\ref{groupdecomposition})
$$\exp(l_{n-1}^{(n-1)}e'_{n-1}+x_{n-1})=\left(exp(\zeta^-_{n-1}e'_{n-1})exp(x_{n-1}')\right)exp(x_{n-1}'')$$ where $x_{n-1}'$ is
a combination of the root vectors $e_i'$, $i<n-1$, and $x_{n-1}''\in \mathfrak n^+\cap w_{n-1}'\mathfrak n^+(w_{n-1}')^{-1}$.
This shows that $l^{(n)}_{n-1}=\zeta^-_{n-1}$. Using the fact that $\mathfrak n^{+(n)}_{n-1}$ is a subalgebra, and the induction
hypothesis that $l_{n-2}^{(n-1)}=\zeta^-_{n-2}$, the two left most terms can be written as
$$exp(\zeta^-_{n-1}e'_{n-1})exp(x_{n-1}')exp(x_{n-1}'')\exp(\zeta^-_{n-2}e'_{n-2}+x_{n-2})$$
$$=exp(\zeta^-_{n-1}e'_{n-1})\exp(\mathbf l_{n-2}^{(n-1)}e'_{n-2}+\mathbf x_{n-2})$$
where $\mathbf l_{n-2}^{(n-1)}-\zeta^-_{n-1}$ depends only on $\zeta^{+}_n$ and $l^{(n-1)}_{n-1}=\zeta^-_{n-1}$, and $\mathbf x_{n-2}\in \mathfrak n^{+(n)}_{n-2}$. We can now continue this process. This will modify $l^{(n-1)}_j$ by adding in terms
that involve $\zeta^+_n$ and $l^{(n-1)}_k$ for $k>j$. By the induction hypothesis, $l^{(n-1)}_k-\zeta^-_k$ depends only on
the $\zeta^{\pm}_m$ variables for $m>k$. Since $k>j$, this implies that $l^{(n)}_j-\zeta^-_j$ will only depend on the $\zeta^{\pm}_m$ variables for $m>j$, i.e. we have proven the statements in (b) for $l^{(n)}_j$ and $g^{(n)}$.

Part (a) is a direct consequence of the fact that $\mathfrak g$ is graded by height. For part (d), the simply laced case, the basic fact is that the terms in the sum (in (\ref{product6})) will vanish for $k>1$, and the coefficient in the $k=1$ term will be multilinear.

Recall that given a triangular decomposition for $\mathfrak g$, there is a Cartan involution (which we denote by $x\to -x^t$) which is
$-1$ on $\mathfrak h$ and interchanges the $\alpha$ and $-\alpha$ root spaces for a positive root $\alpha$. Given a factorization
$$g=\prod_{1\le j\le \mathbf n}^{\leftarrow}i_{\tau_j}(g(\zeta_j))$$
$$g^t=\prod_{1\le j\le \mathbf n}^{\leftarrow}i_{\tau_{\mathbf n+1-j}}(g(\zeta_{\mathbf n+1-j}^t))$$
where $\zeta^t=(\zeta^-,\zeta^+)^t=(\zeta^+,\zeta^-)$. This is the factorization that corresponds to the reversed ordering of positive roots (the conjugated and reversed reduced
factorization for $w_0$, see (3) of Example \ref{orderingexamples}). This just interchanges the claims about the $\zeta^-$ and the $\zeta^+$ variables.

This completes the proof.

\end{proof}

To partially motivate the formulas in the following lemma and the main theorem, note that there is a triangular factorization
\begin{equation}\label{transpose}g(\zeta)^{-t}:=(g(\zeta)^{-1})^t=\left(\begin{matrix}1+\zeta^-\zeta^+&-\zeta^-\\-\zeta^+&1\end{matrix}\right)
\end{equation}
$$=
\left(\begin{matrix}1&0\\-\frac{\zeta^+}{1+\zeta^-\zeta^+}&1\end{matrix}\right)\left(\begin{matrix}1+\zeta^-\zeta^+&0\\0&(1+\zeta^-\zeta^+
)^{-1}\end{matrix}\right)\left(\begin{matrix}1&-\frac{\zeta^-}{1+\zeta^-\zeta^+}\\0&1\end{matrix}\right))$$
provided that $1+\zeta^-\zeta^+\ne 0$.

\begin{lemma} \label{transposelemma}As in the preceding lemma, suppose that
$$g=\prod_{1\le j\le \mathbf n}^{\leftarrow}i_{\tau_j}(g(\zeta_j))$$
Then
$$g^{-t}=\left(\prod_{1\le j\le \mathbf n}^{\leftarrow}i_{\tau_{j}}(g(\eta_{j}))\right)\prod_{1\le j\le \mathbf n}
(1+\zeta^-_j\zeta^+_j)^{h_{\tau_j}}$$
where $\eta=(\eta^-,\eta^+)\in (\mathbb C^2)^{\mathbf n}$ and
$$ \eta^-_k=-\zeta^+_k\prod_{k<j\le \mathbf n} (1+\zeta^-_j\zeta^+_j)^{-\tau_k(h_{\tau_j})}  \text{ and } \eta^+_k=-\zeta^-_k\prod_{k<j\le \mathbf n} (1+\zeta^-_j\zeta^+_j)^{-\tau_k(h_{\tau_j})}$$
\end{lemma}

\begin{proof}This follows by applying (\ref{transpose}) to each of the factors in the factorization of $g$, then moving the $H$ factors to the right.\end{proof}

Before continuing to the main theorem, it is enlightening to first discuss a naive algorithm for solving for the $\zeta$ variables, based on Lemma \ref{workhorse}. By part (b) of Lemma \ref{workhorse}, $\zeta_{\mathbf n}^-=l_{\mathbf n}$,  $\zeta_{\mathbf n-1}^-=l_{\mathbf n-1}$, and in general, using downward induction, $\zeta_k^-=l_k$ plus a polynomial in the variables $l_j$ and $\zeta_j^{+}$ for $j>k$. We substitute these expressions for the $\zeta^-$ variables into the original polynomial expressions of the $u_k$ in terms of $\zeta$, so that $u_k$ is a polynomial in the variables $l_j$ and $\zeta_j^{+}$. By part (c)  $u_1=\zeta_1^+, u_2=\zeta_2^+$, and in general $u_k=\zeta_k^+$ plus a polynomial in the $l$ variables and $\zeta_j^{\pm}$ for $j<k$; the crux of the matter is that this polynomial can depend on $\zeta_k^+$. For example, in Example \ref{A3}, $u_3=\zeta_3^+-\zeta_1^-u_2$, and $\zeta_1^-$ does depend (linearly) on $\zeta_3^+$. For $u_3$, it is easy to see that the dependence on $\zeta_3^+$ is linear, and hence one can solve for $\zeta_3^+$ as a rational function of the $u$, $l$, and $\zeta_j^+$ variables for $j>3$. We now substitute these rational expressions into the $u_k$ for $k\ge 4$. Now $u_4=\zeta_4^+$ plus a rational expression, and in the process of clearly denominators, it is not a priori clear that the resulting polynomial equation depends linearly on $\zeta_4^+$. In experiments this always (miraculously) work out, but we are not able to justify this. Thus from lemma \ref{workhorse} we can see clearly that the $\zeta$ variables are algebraic functions of $l,u$, but it is not clear that they are rational.

\begin{theorem}\label{rationalinverse2} Consider the forward map
$$ F:\mathbb (C^{2})^{\mathbf n} \to N^-N^+:\zeta \to g=\iota_{\tau_{\mathbf n}}(g(\zeta_{\mathbf n}))..\iota_{\tau_1}(g(\zeta_1)) h$$

(a) The map $F$ has a rational inverse, i.e. the $\zeta$ variables are rational functions of
the ordered exponential coordinates.

(b) The map $F$ induces a bijective correspondence between $\{\zeta:det(\partial F)\ne 0\}$ and the domain of the rational inverse in part (a).
\end{theorem}

\begin{proof} The proof involves examining a second algorithm which involves solving for the pairs
$\zeta_k^{\pm}$ in reverse order.

We will initially suppose that $1+\zeta_k^-\zeta_k^+\ne 0$ for all $k$ (We will later sharpen this). Let $h=\prod_{1\le j\le \mathbf n}
(1+\zeta^-_j\zeta^+_j)^{h_{\tau_j}}$. By Lemma \ref{transposelemma}, $g^{-t}h^{-1}=F(\eta)$, and hence by Proposition \ref{forwardmap}
has a unique triangular factorization, $g^{-t}h^{-1}=l'u'$. The ordered exponential coordinates (which we denote by $l'_j,u'_j$) are rational functions of $l,u$, because we can obtain the triangular factorization in the following way. We first factor
$g^{-t}=l^{-t}u^{-t}=l''d''u''$. These factors are rational functions of the $l_j,u_j$. Then $h=d''$, $l'=l''$, and $u'=hu''h^{-1}$.
Hence the ordered exponential coordinates for $l',u'$ are rational functions of the ordered exponential coordinates for $l,u$.

We now essentially repeat the algorithm discussed above, except now we use the exponential coordinates for both $g$ and $g^{-t}$.
At the first step $\zeta_{\mathbf n}^-=l_{\mathbf n}$ and $\zeta_{\mathbf n}^+=-\eta_{\mathbf n}^-=-l_{\mathbf n}'$. At the second step
$\zeta_{\mathbf n-1}^-=l_{\mathbf n-1}$ and $\eta_{\mathbf n-1}^-=l_{\mathbf n-1}'$. By Lemma \ref{transposelemma} $\eta_{\mathbf n-1}^-=-\zeta_{\mathbf n-1}^+((1+\zeta^-_{\mathbf n}\zeta^+_{\mathbf n})^{-\tau_{\mathbf n-1}(h_{\tau_{\mathbf n}})})$. Thus  $\zeta^{\pm}_{mathbf n-1}$ are rational functions of ordered exponential coordinates.

We now continue using downward induction. Suppose that for $ k<j$, $\zeta_{j}^{\pm}$ are rational functions of ordered exponential coordinates. Lemma \ref{workhorse}, and Lemma \ref{transposelemma} for the second equation, imply that
$$\zeta_k^-=l_k+p \text{ and }-\zeta^+_k\prod_{k<j\le \mathbf n} (1+\zeta^-_j\zeta^+_j)^{-\tau_k(h_{\tau_j})}=\eta_k^-=l_k'+ p' $$
where $p$ ($p'$) is a polynomial in $\zeta_j^{\pm}$ ($\eta_j^{\pm}$, respectively) for $k<j\le \mathbf n$. By the induction hypothesis
$\zeta_{k}^{\pm}$ are rational functions of ordered exponential coordinates. This completes the proof of part (a).

It follows from part (a) that the forward map $F$ induces a bijective correspondence between $U$, the $F$ inverse image of the domain of its rational inverse, and the domain $D\subset N^-\times N^+$ of its rational inverse. From the algorithm for the $\zeta$ variables, we see that $U$ contains the nonvanishing set of
$$det(\partial F)=\prod_{1\le k< j\le \mathbf n}(1+\zeta^-_j\zeta^+_j)^{\tau_k(h_{\tau_j})}$$
$U$ must be contained $\{det(\partial F)\ne 0\}$, because local injectivity implies nonvanishing of the determinant.
This implies part (b).

\end{proof}

\section{Haar Measure in Root Subgroup Coordinates}\label{haar}

Let $d\lambda_G$, $d\lambda_{N^{\pm}}$, and $d\lambda_H$ denote Haar measures for $G$, $N^{\pm}$ and $H$, respectively. These measures are biinvariant and essentially unique.
In triangular coordinates
$$d\lambda_G(g)=a^{4\delta}  d\lambda_{N^-}(l)d\lambda_{H}(h)d\lambda_{N^+}(u) $$
where $g=lhu$, $h=ma$ relative to $H=T \times A$, and $\delta$ is half the sum of the positive complex roots; see e.g.
Proposition 5.21 of \cite{Helgason}.
For example for $SL(2,\mathbb C)$
$$g=\left(\begin{matrix}1&0\\l_1&1 \end{matrix}\right)\left(\begin{matrix}m_0a_0&0\\0&(m_0a_0)^{-1} \end{matrix}\right)\left(\begin{matrix} 1&u_1\\0&1\end{matrix}\right) $$
$$d\lambda_G(g)=a_0^{4}  d\lambda(l_1)d\lambda(m_0a_0)d\lambda(u_1) $$

\begin{remarks} (a) As a reminder of why this is the correct formula, suppose that we multiply $g=lmau$ on the right by $a_1\in A$.
Then $ga_1=lm(aa_1)u^{a_1^{-1}}$ and
$$d\lambda_G(ga_1)=(aa_1)^{4\delta}  d\lambda_{N^-}(l)d\lambda_{H}(ha_1)d\lambda_{N^+}(u^{a_1^{-1}})  =d\lambda(g)$$

(b) If we factor $g=luh$ or $g=hlu$, then
$$d\lambda_G(g)= d\lambda_{N^-}(l)d\lambda_{N^+}(u)d\lambda_{H}(h) $$
i.e. the density disappears.
\end{remarks}

Choose a reduced factorization of the longest Weyl group element,
$$w_0=r_{\mathbf n}...r_1$$
and consider the associated root subgroup coordinates for a Zariski open subset of $\Sigma_1$ and of $G$,
$$ H\times (\mathbb C^{2})^{\mathbf n} \to G:(h;\zeta_1,...,\zeta_{\mathbf n})\to g=i_{\tau_{\mathbf n}}(g(\zeta_{\mathbf n}))...i_{\tau_1}(g(\zeta_1))h$$

\begin{theorem}In terms of root subgroup coordinates a Haar measure for $G$ is
$$d\lambda_G(g)=\prod_{1\le j\le \mathbf n}|1+\zeta^-_j\zeta^+_j|^{2(\delta(h_{\tau_j})-1)} d\lambda(\zeta)d\lambda_H(h)  $$
where $d\lambda(\zeta)$ is Lebesgue measure, $d\lambda_H(h)$ is Haar measure for $H$, $\delta$ is half the sum of the positive roots,
and $h_{\tau}$ is the coroot corresponding to a positive root $\tau$.
In simply laced cases
$$d\lambda_G(g)=\prod_{1\le j\le \mathbf n}|1+\zeta^-_j\zeta^+_j|^{2(ht(\tau_j)-1)} d\lambda(\zeta)d\lambda_H(h)  $$

\end{theorem}

In terms of root subgroup coordinates, $ma(g)=h$ and does not depend at all on $\zeta$. Thus the density $a^{4\delta}$ is unity. Also Haar measure for $N^-$ ($ N^+$) is Lebesgue measure in terms of the ordered exponential coordinates $l_j$ ($u_j$, respectively) of the previous section. Consequently the problem reduces to understanding the pullback of Lebesgue measure $\prod_{1\le j\le \mathbf n} d\lambda(l_j)d\lambda(u_j)$ to root subgroup coordinates. Thus the theorem can be restated in the following way (or as in the introduction, see (\ref{detformula})).

\begin{theorem}\label{detjacobian} In terms of root subgroup coordinates the holomorphic volume form $dl_{\mathbf n} \wedge ...\wedge dl_1\wedge du_{\mathbf n} \wedge ... \wedge du_1$
equals
$$\prod_{1\le j\le \mathbf n}(1+\zeta^-_j\zeta^+_j)^{\delta(h_{\tau_j})-1} d\zeta^-_{\mathbf n}\wedge ...\wedge d\zeta^-_1 \wedge d\zeta^+_{\mathbf n}\wedge ...\wedge d\zeta^+_1$$
\end{theorem}

\begin{proof}The algorithm in the proof of Theorem \ref{rationalinverse2} shows that
$$det(\partial F)=\prod_{1\le k< j\le \mathbf n}(1+\zeta^-_j\zeta^+_j)^{\tau_k(h_{\tau_j})}$$
The formula in the theorem is a consequence of Lemma 3.3 of \cite{CP}, which is the following statement.

\begin{lemma}\label{alglemma} For $j=2,...,\mathbf n$,
$$\delta(h_{\tau_j})-1=\sum_{k=1}^{j-1} \tau_k(h_{\tau_j}).$$
\end{lemma}

\end{proof}

\subsection{A Reformulation}

There is an alternate way to formulate root subgroup coordinates which directly generalizes the unitary case.
Consider the two different versions of root subgroup factorization
\begin{equation}\label{newgintopstratum}
g= \prod_{1\le j\le \mathbf n}^{\leftarrow}\iota_{\tau_{j}}(\left(\begin{matrix}1&0\\\zeta_j^-&1\end{matrix}\right)\left(\begin{matrix}1&\zeta_j^+\\0&1\end{matrix}\right) ) h
\end{equation}
(the version we are considering in this paper)
and
\begin{equation}\label{oldgintopstratum}
g= \prod_{1\le j\le \mathbf n}^{\leftarrow}\iota_{\tau_{j}}(\left(\begin{matrix}1&0\\\eta_j^-&1\end{matrix}\right)\left(\begin{matrix}\mathbf a(\eta_j)&0\\0&\mathbf a(\eta_j)\end{matrix}\right)\left(\begin{matrix}1&\eta_j^+\\0&1\end{matrix}\right)) \tilde h\end{equation}
where $\mathbf a(\zeta):=(1-\zeta^-\zeta^+)^{-1/2}$ involves a choice of square root (the version which directly generalizes the unitary case).
Then
$$\zeta_j^-=\prod_{j<k\le \mathbf n}\mathbf a(\eta_k)^{-\tau_j(h_{\tau_k})} \eta_j^- \text{ and } \zeta_j^+=\prod_{j\le k\le \mathbf n}\mathbf a (\eta_k)^{\tau_j(h_{\tau_k})} \eta_j^+ $$
and
$$h=\prod_{1\le k\le \mathbf n}\mathbf a(\eta_k)^{h_{\tau_k}}\tilde h $$
This implies
$$\zeta_j^-\zeta_j^+=\mathbf a (\eta_j)^{2}\eta_j^-\eta_j^+ \text{ and } 1+\zeta_j^-\zeta_j^+=(1-\eta_j^-\eta_j^+)^{-1}=\mathbf a(\eta_j)^2$$

Now we change variables in the Haar measure formula:
$$d\lambda_G(g)=\prod_{1\le j\le \mathbf n}|1+\zeta^-_j\zeta^+_j|^{2(\delta(h_{\tau_j})-1)} d\lambda(\zeta)d\lambda_H(h)  $$
$$=\prod_{1\le j\le \mathbf n}|1-\eta^-_j\eta^+_j|^{-2(\delta(h_{\tau_j})-1)} \prod_{1\le j\le \mathbf n}\mathbf a(\eta_j)^4 d\lambda(\eta)d\lambda_H(\tilde h)  $$
$$=\prod_{1\le j\le \mathbf n}|1-\eta^-_j\eta^+_j|^{-2\delta(h_{\tau_j}))} d\lambda(\eta)d\lambda_H(\tilde h)= \prod_{1\le j\le \mathbf n}|\mathbf a(\eta_j)|^{4\delta(h_{\tau_j}))} d\lambda(\eta)d\lambda_H(\tilde h) $$

This implies that with respect to the second version of root subgroup coordinates
$$d\lambda(l)d\lambda(u)=d\lambda(\eta)$$

\section{Canonical Orderings in Classical Cases}\label{canonicalorderings}

In this section we identify preferred orderings of positive roots (or reduced factorizations of longest Weyl group elements) in
classical cases by considering inductive limits, e.g. $GL(\infty)=\lim_{n\to \infty} GL(n)$. In the general linear case this is obviously reasonable because the inductive limit is compatible with triangular factorization, in the following sense: If $g$ is an $N\times N$ matrix, and $g^{(n)}$ denotes the principal $n\times n$ minor, then $g=ldu$ implies that $g^{(n)}=l^{(n)}d^{(n)}u^{(n)}$.

In reference to the associated Weyl groups, e.g. $S_{\infty}=\lim_{n\to \infty} S_n$,
the infinite limit does not have an element of longest length.
However it makes sense to seek reduced sequences of simple reflections such that for each $n$, the finite subsequence of appropriate
length corresponds to a reduced factorization for the Weyl group element of longest length for the $n$th subgroup, e.g. $GL(n)$.
In all of the examples that follow, $\mathfrak h$ consists of diagonal matrices (with some symmetry), and
$\mathfrak n^+$ consists of upper triangular matrices. Our conventions are consistent with those in part II of \cite{Pi1}.

In the first four subsections we present what we will refer to as the canonical orderings in the classical cases. In the last subsection
we consider $GL(2\infty)$ and observe that if the rank increases by more than $1$, then uniqueness is lost.

\subsection{$A_{\infty}$: $GL(1,\mathbb C) \subset GL(2,\mathbb C) \subset ...\subset GL(\infty,\mathbb C)$}

Consider the standard ordered basis $\epsilon_1,\epsilon_2,\epsilon_3,...$
for $\mathbb C^{\infty}$ and the associated inclusions of general linear groups
$$GL(1,\mathbb C)\subset GL(2,\mathbb C) \subset ...\subset GL(\infty,\mathbb C)$$
and their Weyl groups
$$S_1 \subset S_2 \subset ...\subset S_{\infty}$$
Also let $\lambda_1,\lambda_2,...$ denote the basis for $\mathfrak h^*$ which dual to the basis $\epsilon_{1}\otimes \epsilon_{1}^*,\epsilon_{2}\otimes \epsilon_{2}^*,...$ for $\mathfrak h$.

We seek a reduced sequence of reflections corresponding to simple roots (i.e. the adjacent transpositions), $r_1,r_2,...$, such that for each $n$,
$$ r_{\frac{n(n-1)}{2}}...r_2r_1$$
is equal to the longest element $w_0^{(n)}\in S_n$, i.e. the permutation
$$w_0^{(n)}=\left(\begin{matrix}1&2&...&n\\n&n-1&...&1\end{matrix}\right)$$

\begin{proposition}There is a unique reduced sequence of reflections corresponding to simple roots which is compatible with the
standard inclusions above. If $s_i$ denotes the transposition of $i$ and $i+1$, then the sequence is given by
$$s_1,s_2,s_1,s_3,s_2,s_1,s_4,s_3,s_2,s_1,... $$
The associated ordering of positive roots is the lexicographic pattern
$$\left(\begin{matrix} \lambda_1& \tau_1&\tau_2&\tau_4&\tau_7&..\\ & \lambda_2 &\tau_3&\tau_5&\tau_8&..\\ & & \lambda_3 & \tau_6&\tau_9&..\\ &  & &\lambda_4&\tau_{10}&..\\ &  & & &\lambda_5&..\\.. & .. & ..& ..&..&..\end{matrix}\right) $$
i.e.
$$\tau_1=\lambda_1-\lambda_2,\quad \tau_2=\lambda_1-\lambda_3,\quad \tau_3=\lambda_2-\lambda_3,\quad \tau_4=\lambda_1-\lambda_4,\quad \tau_5=\lambda_2-\lambda_4,...$$

\end{proposition}

\begin{proof} $S_n$ is the group generated by the $s_i$ and the relations $s_i^2=1$,
$s_is_{i+1}s_i=s_{i+1}s_is_{i+1}$ (the braid relation), and $s_is_j=s_js_i$ if $|i-j|>1$ (locality). It is well-known that one can go from one reduced factorization of a group element $w$ to another using the braiding and locality relations, i.e. the graph with vertices consisting of reduced factorizations of $w$ and with edges given by these relations is connected (this is true in general for groups with a Coxeter system of relations). Using this one can check that the $n$ cycle $(12..n)$ has length $n-1$ and there is a unique reduced factorization
$(12..n)=s_1..s_{n-1}$. Consequently in the proposition we are asserting that $w_0^{(n)}$ is the composition of cycles
$w_0^{(n)}=(12...n)...(123)(12)$.

It is easily checked that the given sequence of simple reflections corresponds to the lexicographic pattern of positive roots (which is a second way of seeing that the sequence is compatible with the inclusions and that the appropriate truncations of the sequence correspond to reduced factorizations of the $w_0^{(n)}$).

For small $n$ uniqueness is easily checked by hand. Suppose that uniqueness holds for $n$ and consider $S_{n+1}$. The question is whether, after the sequence corresponding to $w_0^{(n)}$, the sequence has to continue as $s_{n+1},...,s_1$. Since the indices are adjacent, it is not possible to alter this sequence in any way (Of course one could interchange $s_1$ and $s_{n+1}$, but this would destroy the compatibility with the inclusions). This implies uniqueness, and completes the proof.
\end{proof}

\subsection{$B_{\infty}$: $O(3,\mathbb C)\subset O(5,\mathbb C)\subset ...\subset O(2\infty+1,\mathbb C)$}

We consider the vector space with (doubly infinite) ordered basis
$$...,\epsilon_n,...,\epsilon_0,...,\epsilon_{-n},... $$
with the symmetric form for which $(\epsilon_n,\epsilon_{-n})=1$ and all other pairings vanish.
The Lie algebra $o(2\infty+1,\mathbb C)$ is realized as matrices which are skew-symmetric with respect
to the anti-diagonal. We consider the (double infinite) inclusions
$$O(3,\mathbb C)\subset O(5,\mathbb C) \subset ...$$

A basis for $\mathfrak h$ compatible with this filtration is
$$\epsilon_n\otimes\epsilon_n^*-\epsilon_{-n}\otimes\epsilon_{-n}^*, \qquad n=1,2,3,...$$
Let $\lambda_1,\lambda_2,...$ denote the dual basis. The simple positive roots are
$$\alpha_1=\lambda_1,\quad  \alpha_2=\lambda_2-\lambda_1,\quad \alpha_3=\lambda_3-\lambda_2,... $$
The Weyl group (the group of signed permutations), viewed as linear transformations of $\mathfrak h^*$,  is realized as follows: $s_1(\lambda_1)=-\lambda_1$ and $s_1(\lambda_n)=\lambda_n$, $n>1$, and for $n>1$, $s_n$ transposes $\lambda_n$ and $\lambda_{n-1}$, and fixes the other $\lambda_j$.

\begin{proposition}There is a unique reduced sequence of reflections corresponding to simple roots, $r_1,r_2,...$, such that for each $n$,
$$ r_{\frac{2n(2n-1)}{2}}...r_2r_1$$
is equal to the longest element of the Weyl group of $O(2n+1,\mathbb C)$. The reduced sequence is given by
$$s_1,s_2,s_1,s_2,s_3,s_2,s_1,s_2,s_3,s_4,... $$
The associated ordering of positive roots is lexicographic:
$$\left(\begin{matrix}\lambda_3 & \tau_9&\tau_8 &\tau_7 &\tau_6 &\tau_5 & 0\\   & \lambda_2&\tau_4 &\tau_3 &\tau_2 & 0& \\ & &\lambda_1 &\tau_1 &0 & & \\ & & & 0& & & \\ & &0 & & -\lambda_1& & \\ &0 & & & & -\lambda_2& \\ 0& & & & & & -\lambda_3 \end{matrix} \right)  $$
(we have not indicated the skew reflection across the antidiagonal for the $\tau$s), i.e.
$$\tau_1=\lambda_1,\quad \tau_2=\lambda_1+\lambda_2,\quad \tau_3=\lambda_2,\quad \tau_4=\lambda_2-\lambda_1,...  $$
\end{proposition}

\begin{proof} It is easily checked that the given sequence of simple reflections corresponds to the lexicographic pattern of positive roots (which is a second way of seeing that the sequence is compatible with the inclusions and that the appropriate truncations of the sequence correspond to reduced factorizations of the $w_0^{(n)}$). The relations for the Weyl group are the same as for $S_r$, with one difference: in place of the braid relation for $s_1,s_2$, one has the relation $s_1s_2s_1s_2=s_2s_1s_2s_1$. One can now prove uniqueness
as in the general linear case.
\end{proof}

\subsection{$D_{\infty}$: $O(2,\mathbb C)\subset O(4,\mathbb C)\subset ...\subset O(2\infty,\mathbb C)$}

We consider the vector space with (doubly infinite) ordered basis
$$...,\epsilon_n,...,\epsilon_{-n},... $$
where $n$ ranges over half-integers,
with the symmetric form for which $(\epsilon_n,\epsilon_{-n})=1$ and all other pairings vanish.
The Lie algebra $o(2\infty,\mathbb C)$ is realized as matrices which are skew-symmetric with respect
to the anti-diagonal.
We consider the (double infinite) inclusions
$$O(2,\mathbb C)\subset O(4,\mathbb C) \subset ...$$
A basis for $\mathfrak h$ compatible with this filtration is
$$\epsilon_n\otimes\epsilon_n^*-\epsilon_{-n}\otimes\epsilon_{-n}^*, \qquad n=1,2,3,...$$
Let $\lambda_1,\lambda_2,...$ denote the dual basis. The simple positive roots are
$$\alpha_1=\lambda_1+\lambda_2,\quad  \alpha_2=\lambda_2-\lambda_1,\quad \alpha_3=\lambda_3-\lambda_2,... $$
The Weyl group, viewed as linear transformations of $\mathfrak h^*$,  is realized as follows: $s_1(\lambda_1)=-\lambda_2$, $s_1(\lambda_2)=-\lambda_1$, and $s_1(\lambda_n)=\lambda_n$, $n>2$;  $s_2(\lambda_1)=\lambda_2$, $s_1(\lambda_2)=\lambda_1$, and $s_2(\lambda_n)=\lambda_n$, $n>2$; and for $n>2$, $s_n$ transposes $\lambda_n$ and $\lambda_{n-1}$, and fixes the other $\lambda_j$.

\begin{proposition}There is a (conjecturally unique) reduced sequence, compatible with the inclusions above, which is given by
$$s_1,s_2,s_1,s_2,s_3,s_2,s_1,s_3,s_2,... $$
The associated ordering of roots is lexicographic $$
\left(\begin{matrix}\lambda_4 & & &.. &.. &\tau_8 &\tau_7 & 0\\   & \lambda_3&\tau_6 &\tau_5 &\tau_4 &\tau_3 &0 & \\ & &\lambda_2 &\tau_2 &\tau_1 &0 & & \\ & & &\lambda_1 & 0& & & \\
 & & &0 & -\lambda_1& & & \\ & & 0& & &-\lambda_2& & \\ & 0& & & & & -\lambda_3& \\0 & & & & & & & -\lambda_4 \end{matrix}\right) $$
(we have not indicated the skew reflection across the antidiagonal for the $\tau$s).

\end{proposition}

\begin{proof} It is easily checked that the given sequence of simple reflections corresponds to the lexicographic pattern of positive roots (which is a second way of seeing that the sequence is compatible with the inclusions and that the appropriate truncations of the sequence correspond to reduced factorizations of the $w_0^{(n)}$). The relations for the Weyl group are the same as for $S_r$, with the following exceptions: $s_1$ and $s_2$ commute (rather than braid), and $s_1,s_3$ braid (rather than commute). Using this one checks uniqueness as in the general linear case.

\end{proof}

\subsection{$C_{\infty}$: $Sp(1,\mathbb C)\subset Sp(2,\mathbb C)\subset ...\subset Sp(\infty,\mathbb C)$}

We consider the vector space with (doubly infinite) ordered basis
$$...,\epsilon_n,...,\epsilon_{-n},... $$
where $n$ ranges over half-integers,
with the skew-symmetric form for which $\omega(\epsilon_n,\epsilon_{-n})=1$ and all other pairings vanish.
The Lie algebra $sp(2\infty,\mathbb C)$ is realized as matrices which in block form $\left(\begin{matrix}A&B\\C&D\end{matrix}\right)$
have the property that $\left(\begin{matrix}A&0\\0&D\end{matrix}\right)$ ($\left(\begin{matrix}0&B\\C&0\end{matrix}\right)$) is skew-symmetric (symmetric, respectively) with respect to the anti-diagonal.
We consider the (double infinite) inclusions
$$Sp(2,\mathbb C)\subset Sp(4,\mathbb C) \subset ...$$
A basis for $\mathfrak h$ compatible with this filtration is
$$\epsilon_n\otimes\epsilon_n^*-\epsilon_{-n}\otimes\epsilon_{-n}^*, \qquad n=\frac12,\frac32,...$$
Let $\lambda_1,\lambda_2,...$ denote the dual basis. The simple positive roots are
$$\alpha_1=2\lambda_1,\quad  \alpha_2=\lambda_2-\lambda_1,\quad \alpha_3=\lambda_3-\lambda_2,... $$
Viewed as linear transformations of $\mathfrak h^*$,
$s_1(\lambda_1)=-\lambda_1$ and $s_1(\lambda_n)=\lambda_n$, $n>1$, and for $n>1$, $s_n$ transposes $\lambda_n$ and $\lambda_{n-1}$, and fixes the other $\lambda_j$.

The Weyl group, and the associated inclusions, are identical to the odd orthogonal case. Consequently all that remains is to indicate the pattern for the ordering of the positive roots (it is not the obvious lexicographic ordering!).

\begin{proposition}There is a (conjecturally unique) reduced sequence, compatible with inclusions above, which is given by
$$s_1,s_2,s_1,s_2,s_3,s_2,s_1,s_2,s_3,s_4,...$$
The associated ordering of roots is
 $$\left(\begin{matrix}\lambda_4 &\tau_{16} &\tau_{15} &\tau_{14} &\tau_{12} &\tau_{11} &\tau_{10} &\tau_{13} \\   & \lambda_3&\tau_9 &\tau_8 &\tau_6 &\tau_5 &\tau_7 & \\ & &\lambda_2 &\tau_4 &\tau_2 &\tau_3 & & \\ & & &\lambda_1 &\tau_1 & & & \\
 & & & & -\lambda_1& & & \\ & & & & &-\lambda_2& & \\ & & & & & & -\lambda_3& \\ & & & & & & & -\lambda_4 \end{matrix}\right) $$
(we have not indicated the appropriate reflection across the antidiagonal for the $\tau$s), i.e.
$$\tau_1=2\lambda_1,\quad\tau_2=\lambda_2+\lambda_1,\quad\tau_3=2\lambda_2,\quad\tau_4=\lambda_2-\lambda_1,$$
$$\quad\tau_5=\lambda_2+\lambda_3,  ,\quad \tau_6=\lambda_1+\lambda_3,\quad \tau_7=2\lambda_3,\quad \tau_8=\lambda_3-\lambda_1,\quad \tau_9=\lambda_3-\lambda_2,$$
$$\tau_{10}=\lambda_{4}+\lambda_3,\quad \tau_{11}=\lambda_{4}+\lambda_2,
\quad \tau_{12}=\lambda_{4}+\lambda_1,\quad \tau_{13}=2\lambda_4,\quad \tau_{14}=\lambda_{4}-\lambda_1, ... $$

\end{proposition}

\begin{proof}The type $C$ Weyl group is the same as for type $B$. So the only thing that needs to be checked is the claim about the ordering of the positive roots, which is routine.\end{proof}

\subsection{$A_{2\infty}$: $GL(2,\mathbb C)\subset GL(4,\mathbb C)\subset...\subset GL(2\infty,\mathbb C)$}

This case is more a counter-example than an example. The rank increases by two at each stage, and somewhat surprisingly,
this allows for the possibility of infinitely many compatible orderings.

Consider the doubly infinite basis
$$...,\epsilon_{3/2},\epsilon_{1/2},\epsilon_{-1/2},\epsilon_{-3/2} \subset ...$$
and the associated inclusions
$$GL(2,\mathbb C)\subset GL(4,\mathbb C) \subset ...$$
and
$$S_2 \subset S_4 \subset ...$$
For an integer $j$, let $s_j$ denote the transposition of $-\frac12+j$ and $\frac12+j$.

We seek a reduced sequence of reflections corresponding to simple roots, $r_1,r_2,...$, such that for each $n$,
$$ r_{\frac{2n(2n-1)}{2}}...r_2r_1$$
is equal to the longest element of the Weyl group of $GL(2n,\mathbb C)$, i.e. the permutation
$$\left(\begin{matrix}-n/2&-n/2+1&...&n/2-1&n/2\\n/2&n/2-1&...&-(n/2-1)&-n/2\end{matrix}\right)$$

\begin{proposition} (a) There is a reduced sequence which is given by
$$s_0, \quad s_1,s_0,s_{-1},s_0, s_1, \quad s_2,s_1,s_0,s_{-1},s_{-2},s_{-1}, s_0,s_1,s_2,\quad s_{3},... $$
The associated ordering of roots has the following pattern
$$\left(\begin{matrix}0 &\tau_{15 } & \tau_{14 }& \tau_{13 }&\tau_{12 } &\tau_{11 } \\ &0 & \tau_{6 }&\tau_{5 } &\tau_{4 } & \tau_{7 }\\
& &0 & \tau_1&\tau_2&\tau_{8 }\\. & & &0 &\tau_{3 } &\tau_{9 } \\ & & & &0 & \tau_{10 }\\ & & & & & 0
\end{matrix}\right) $$
where $\tau_1$ is located in the $(-\frac12,\frac12)$ spot, i.e. $\tau_1=\lambda_{1/2}-\lambda_{-1/2}$.

(b) There are infinitely many other sequences which are compatible with the inclusions.
\end{proposition}

\begin{proof} (a)
This follows from the fact that
$$s_ns_{n-1}..s_{-(n-1)}s_{-n}s_{-(n-1)}..s_n$$ transposes $\pm\frac{n}{2}$.

(b) In this expression for the transposition $(-\frac{n}{2},\frac{n}{2})$, we can using the Weyl group relations to reexpress this transposition in many different reduced ways, e.g. we can replace $s_{-(n-1)}s_{-n}s_{-(n-1)}$ by $s_{-n}s_{-(n-1)}s_{-n}$.
Since we can do this for any $n$, this leads to infinitely many sequences compatible with the inclusions.
\end{proof}

\subsection{Concluding Comment}\label{concludingcomment}

Given an $n\times n$ matrix $M$, let $M^{I}_J$ denote the submatrix
with rows indexed by a sequence $I$ and columns indexed by a sequence $J$.
If $g=(g_{ij})\in GL(n,\mathbb C)$ has a triangular factorization
$$g=l\circ d\circ u$$
then
$$d=diag(\sigma_1,\sigma_2/\sigma_1,\sigma_3/\sigma_2,..,\sigma_
n/\sigma_{n-1})$$
where $\sigma_k$ is the determinant of the $k^{th}$ principal
submatrix, i.e. $\sigma_k=det(g^{(1,...,k)}_{(1,...,k)})$, for $i>j$,
$$l_{ij}=\frac{det(g^{(1,...,j-1,i)}_{(1,...,j)})}{\sigma_j} $$
and for $i<j$
$$ u_{ij}=\frac{det(g^{(1,...,i)}_{(1,...,i-1,j)})}{\sigma_i}$$

For the canonical orderings in classical cases, based on experiments and the (LDU based) algorithm in Section \ref{rational}, it seems very likely to me that there exist (similar, but more complicated) analogues of these formulas for root subgroup coordinates.

\end{document}